\documentclass[12pt]{amsart}

\usepackage[english]{babel}
\usepackage[utf8x]{inputenc}
\usepackage{textcomp}
\DeclareUnicodeCharacter{2217}{\textasteriskcentered}

\usepackage[T1]{fontenc}

\usepackage[margin=1.5in]{geometry}

\usepackage{amsmath}
\usepackage{amstext}
\usepackage{amsfonts}
\usepackage{amssymb}
\usepackage{amsthm}
\usepackage{amsrefs}
\usepackage{mathtools}
\usepackage{dsfont}
\usepackage{color}
\usepackage{graphicx}
\usepackage[colorinlistoftodos]{todonotes}
\usepackage[colorlinks=true, allcolors=blue]{hyperref}
\usepackage{float}
\usepackage{mathrsfs}
\allowdisplaybreaks

\newtheorem*{cor*}{Corollary}
\newtheorem*{thm*}{Theorem}
\newtheorem*{lem*}{Lemma}

\newtheorem*{prop*}{Proposition}
\newtheorem*{rem*}{Remark}

\newtheorem{theo}{Theorem}

\newtheorem{coro}{Corollary}

\theoremstyle{definition}

\newtheorem{rem}{Remark}
\newtheorem*{defn*}{Definition}

\theoremstyle{remark}

\newcommand{\pr}{\operatorname{Prob}}

\newcommand{\act}{\curvearrowright}

\newcommand{\cS}{\mathcal{S}}

\newcommand{\bE}{{\mathbb{E}}}

\newcommand{\bN}{{\mathbb{N}}}

\newcommand{\vf}{{\varphi}}

\newcommand{\Sub}{\operatorname{Sub}}

\usepackage[nodisplayskipstretch]{setspace}

\allowdisplaybreaks

\title[Y. HARTMAN AND M. KALANTAR] 
{Stabilizer Subgroups and the Simplicity of Reduced Crossed Products}

\author[Y. Hartman]{Yair Hartman}
\address{Yair Hartman\\ Ben-Gurion University of the Negev}
\email{hartmany@bgu.ac.il}

\author[M. Kalantar]{Mehrdad Kalantar}
\address{Mehrdad Kalantar\\ University of Oxford}
\email{mehrdad.kalantar@maths.ox.ac.uk}

\begin{document}

\begin{abstract}
Given a minimal action $G\act X$ of a countable group $G$ on a compact space $X$, we prove that if the reduced crossed product $G\ltimes_rC(X)$ is simple, then there exists a point whose stabilizer subgroup has trivial amenable radical. As a consequence, we give a complete characterization of the simplicity of the reduced crossed product of minimal actions of countable linear groups, hyperbolic groups, and, more generally, for groups with countably many amenable subgroups. This answers a question of Ozawa (2014) for these classes of groups.
Furthermore, in the case of an infinite uniformly recurrent subgroup of a $C^*$-simple group, we prove that almost every subgroup has a trivial amenable radical, with respect to a fully supported, atomless probability measure.
\end{abstract} 

\thanks{{\ }\\[-3ex]YH was supported by the funding from the European Research Council (ERC) under the European Union's Seventh Framework Programme (FP7-2007-2013) (Grant agreement
No. 101078193).  MK was supported by the NSF grant DMS-2155162, and The Simons Foundation grant SFI-MPS-TSM-00014307.\\
For the purpose of open access, the authors have applied a CC-BY license to any author accepted manuscript arising from this submission.}

\maketitle



The crossed product $C^*$-algebras associated with dynamical systems were among the very first constructed examples of operator algebras. Since their introduction, understanding the relationship between the dynamical properties of an action and the algebraic structure of the resulting crossed product has remained a central problem in the field.

In particular, an old problem in $C^*$-algebra theory, which is the main focus of this work, is to determine when the reduced crossed product of a minimal action is simple, i.e.\ has no non-zero proper closed ideals.

More precisely, let $G$ be a countable group, $X$ a compact Hausdorff space, and $G\act X$ by homeomorphisms. We are concerned with the question of finding necessary and sufficient conditions, in terms of the stabilizer subgroups, for the 
reduced crossed product $G\ltimes_rC(X)$ to be simple.

In~\cites{Ell80} Elliott proved that $G\ltimes_rC(X)$ is simple if $G\act X$ is topologically free, that is, the set of points with trivial stabilizer is dense. For a minimal action, this is equivalent to the existence of one point with a trivial stabilizer. 
This improved earlier results of Effros--Hahn~\cite{effros1967locally}.
Furthermore, for amenable groups, topological freeness was proved to characterize simplicity of the reduced crossed product by Kawamura and Tomiyama~\cites{KawTom} (see also Archbold--Spielberg~\cite{ArSp} along with Anantharaman-Delaroche~\cite{anantharaman1987systemes} for the setup of topological amenable actions).

However, the above characterization fails miserably for general actions of non-amenable groups. For example, the trivial action of the free group $\mathbb{F}_2$, which is as far as can be from being free, yields the simple reduced $C^*$-algebra of $\mathbb F _2$~\cite{Pow75}.

In~\cite{BKKO}, it was proven that for a boundary action $G\act X$, if $G_x$ is $C^*$-simple for some $x\in X$, then $G\ltimes_rC(X)$ is simple. Ozawa noted in~\cite{Oz} that the same holds for any minimal action $G\act X$, and then asked whether the converse is true: 
\begin{center}
\textit{Does simplicity of $G\ltimes_rC(X)$ imply $C^*$-simplicity of stabilizer subgroups?}
\end{center}
The affirmative answer to this question would consequently give a complete characterization of the simplicity of $G\ltimes_rC(X)$ in terms of the point stabilizers.

In this direction, it was proven in~\cite{Oz} and~\cite{Kaw17} that if $G\ltimes_rC(X)$ is simple, then the neighborhood stabilizer subgroups $G_x^o$ consisting of group elements that pointwise fix an open neighborhood of $x$, are either all trivial or all non-amenable.

While the problem of characterizing simplicity in terms of point stabilizers is formulated in topological terms, our strategy relies on ``state-theoretic'' and probabilistic techniques. We use the non-commutative stationary theory developed by the authors in~\cite{HK1}, and its generalization by Amrutam-Ursu~\cite{AmrUrs}, which provides a characterization of simplicity in terms of stationary states.

Applying this stationary approach to the crossed product yields Theorem~\ref{main}, while its application to the reduced group $C^*$-algebra yields Theorem~\ref{thm:B}.
\\[1ex]
The following is our main result.
\begin{theo}\label{main}
Let $G$ be a countable group, $X$ a compact space, and $G\act X$ by homeomorphisms. If $G\ltimes_rC(X)$ is simple, then there exists a point $x_0\in X$ such that the stabilizer subgroup $G_{x_0} $ has trivial amenable radical.
\end{theo}
Recall that every $C^*$-simple group has trivial amenable radical (\cite{KK,BKKO}), but the converse is not true in general~\cite{LeBoud}.
However, for large classes of groups, including all linear groups, hyperbolic groups, and, more generally,  groups with countably many amenable subgroups, these properties are equivalent (see e.g.~\cite[Theorem 6.12]{BKKO}).

In particular, Theorem~\ref{main} yields a complete characterization of simplicity of the reduced crossed product $C^*$-algebras associated to minimal actions of these classes of groups. This answers the question of Ozawa in these cases.

\begin{coro}\label{cor:lin}
Let $G$ be a countable group that is either linear or has countably many amenable subgroups (e.g., a hyperbolic group), and let $G\act X$ be a minimal action. Then 
$G\ltimes_rC(X)$ is simple
if and only if there exists a point $x_0\in X$ such that the stabilizer subgroup $G_{x_0}$ is $C^*$-simple.
\end{coro}

\begin{rem}
Note that in the statement of Corollary~\ref{cor:lin}, we may replace the stabilizers with open stabilizers.  Indeed, note that if for some $x\in X$, the stabilizer $G_x$ is $C^*$-simple, then $G^o_x$ is also $C^*$-simple, since it is a normal subgroup of $G_x$. Conversely, it was proven in~\cite{Kaw17} that if there exists some point $x\in X$ whose open stabilizer $G^o_x$ is $C^*$-simple, then the reduced crossed product is simple.
\end{rem}


Before we get into the proofs of the above results, let us briefly recall some definitions and fix our notation.

We denote by ${\rm Sub}(G)$ the Chabauty space of $G$, that is, the set of all subgroups of $G$, 
equipped with the product topology. This is a compact $G$-space, where $G\act {\rm Sub}(G)$ by conjugations.
We denote by ${\rm Sub}_{am}(G)$ the subset of all amenable subgroups of $G$, which is closed in ${\rm Sub}(G)$.
For $H\le G$ we denote by $R_{a}(H)$ the amenable radical of $H$. 


We recall the notion of $C(X)$-valued probability measures as defined and studied by Amrutam--Ursu in \cite{AmrUrs}.
Given a compact space $X$, denote by $C(X)^+$ the set of all non-zero positive functions $f\in C(X)$. 
By a \emph{$C(X)$-valued probability measure} on $G$ we mean a map $\mu:G\ni s\mapsto \mu(s)\subset C(X)^+$ such that $\sum_{s\in G}\sum_{\vf\in \mu(s)} \vf(x) = 1$ for every $x\in X$.

We may formally represent $\mu=\sum_{s\in G}\sum_{i\in I_s}\vf_i\delta_s$, where $I_s$ is the index set such that $\mu(s)=\{\vf_i : i\in I_s\}$.

If $G\act X$, then for a state $\eta\in\cS(G\ltimes_rC(X))$, the convolution $\mu*\eta\in \cS(G\ltimes_rC(X))$ is defined by
\[
\mu*\eta(a) = \eta\Big(\sum_{s\in G}\sum_{i\in I_s}\sqrt{\vf_i}\,s\!\cdot\! a\, \sqrt{\vf_i}\,\Big)
\]
for all $a\in G\ltimes_rC(X)$. The state $\eta\in \cS(G\ltimes_rC(X))$ is said to be $\mu$-stationary if $\mu*\eta=\eta$.

Note that in the case of the trivial $G$-action, the above notion coincides with $\mu$-stationary states on the reduced $C^*$-algebra $C^*_r(G)$ of $G$ as introduced and studied in~\cite{HK1}.
\\[1ex]
\noindent
{\it Proof of Theorem \ref{main}.}~
Let $X$ be a compact space, let $G\act X$ by homeomorphisms, and assume $G\ltimes_rC(X)$ is simple. First, note that this implies that $G\act X$ is a minimal action: indeed, a non-trivial invariant open subset gives a non-zero proper ideal of the crossed product.

For each $x\in X$, define $\eta_x$ to be the state on the reduced crossed product $G\ltimes_rC(X)$ such that  $\eta_x(\lambda_g \vf)= \vf(x) 1_{R_a(G_x)}(g)$ for $g\in G$ and $\vf\in C(X)$. Recall that $R_a(G_x)$ is the amenable radical of the stabilizer subgroup $G_x$. As
The reason this defines a state on the \textit{reduced }crossed product is thanks to the amenability of $R_a(G_x)$.

Fix some $g \in G$ and $\psi\in C(X)$. We claim that the map $x\mapsto \eta_x(\lambda_g \psi)$ is Borel. For this, it suffices to show that the set $\{H\in \Sub(G) : g\in R_a(H)\}$ is closed in the Chabauty topology, and then, since the stabilizer map $x\mapsto G_x$ is Borel, and $\psi$ is continuous, the claim follows.

So, let $(H_n)_{n\in\bN}$ be a sequence in $\Sub(G)$ with $g\in R_a(H_n)$ for all $n\in\bN$, and assume that $H_n\to H\in {\rm Sub}(G)$. By passing to a subsequence, we may also assume $R_a(H_n)\to N$ for some $N\le H$, and by the definition of the Chabauty topology, $g\in N$.
Since the set of amenable subgroups of $G$ is a closed subset of $\Sub(G)$~\cite{caprace2014relative}, it follows that $N$ is amenable as well. 

Furthermore, for any fixed $t\in N$ and $h\in H$, there is some $n_0\in\bN$ such that $t\in R_a(H_n)$ and $h\in H_n$ for all $n\ge n_0$. Then $h th^{-1}\in R_a(H_n)$ for all $n\ge n_0$, therefore $h th^{-1}\in N$, hence $N$ is normal in $H$. Since $N$ is amenable and normal we get that $N\le R_a(H)$, which implies that $g\in R_a(H)$. This concludes the proof that the set $\{H\in \Sub(G) : g\in R_a(H)\}$ is closed, and so implies that the map $x\mapsto \eta_x(\lambda_g \psi)$ is Borel, for fixed $g \in G$ and $\psi\in C(X)$.

Since $G\ltimes_rC(X)$ is simple, we get by~\cite[Corollary 1.4]{AmrUrs}, a $C(X)$-valued probability measure $\mu=\sum_{s\in G}\vf_s\delta_s$ such that every $\mu$-stationary state $\eta$ on $G\ltimes_rC(X)$ comes from a probability measure on $X$; that is, if we denote by $\bE:G\ltimes_rC(X)\to C(X)$ the canonical conditional expectation defined by $\bE(\lambda_t \vf)=0$ for every $\vf\in C(X)$ and $e\neq t\in G$, then every $\mu$-stationary state $\eta$ on $G\ltimes_rC(X)$ is of the form $\eta=\eta|_{C(X)}\circ \bE$.

Now, choose a $\mu$-stationary probability measure $\nu\in \pr(X)$. 
Notice that here we mean \textit{$\mu$-stationary} in the sense of \cite{AmrUrs}, as $\mu$ is not a random walk on the group, but a $C(X)$-valued probability measure. Since the map $x \mapsto \eta_x(\lambda_g \psi)$ is Borel for fixed $g\in G$ and $\psi\in C(X)$, we may define a state $\eta$ on $G\ltimes_rC(X)$ by $\eta=\int_X\eta_xd\nu(x)$.

Next, we show that $\eta$ is $\mu$-stationary. For that, it is enough to show that $\mu * \eta (\lambda_g \psi ) =\eta(\lambda_g \psi)$ for every $g\in G$ and $\psi\in C(X)$. Indeed,  fix such $g$ and $\psi$, then we have 
\[\begin{split}
\mu*\eta(\lambda_g \psi)
&=
\eta\big(\sum_{s\in G}\sum_{i\in I_s} \sqrt{\vf_i}\,\lambda_{s^{-1}gs}\, s\!\cdot\! \psi\, \sqrt{\vf_i}\,\big)\
\\&=
\int_X\sum_{s\in G}\sum_{i\in I_s} \eta_x( \sqrt{\vf_i}\,\lambda_{s^{-1}gs}\, s\!\cdot\! \psi\, \sqrt{\vf_i})d\nu(x)
\\&=
\int_X\sum_{s\in G}\sum_{i\in I_s} \vf_i(x) \psi(sx)1_{R_a(G_x)}(s^{-1}gs)d\nu(x)
\\&=
\int_X\sum_{s\in G}\sum_{i\in I_s} \vf_i(x) \psi(sx)1_{R_a(sG_xs^{-1})}(g)d\nu(x)
\\&=
\int_X\sum_{s\in G}\sum_{i\in I_s} \vf_i(x) \psi(sx)1_{R_a(G_{sx})}(g)d\nu(x)
\\&=
\int_X\psi(x)1_{R_a(G_x)}(g)\,d\mu*\nu(x)
\\&=
\int_X\psi(x)1_{R_a(G_x)}(g)\,d\nu(x)
\\&=
\int_X\eta_x(\lambda_g \psi)d\nu(x)
\\&=
\eta(\lambda_g \psi)\,,
\end{split}\]
which shows that $\eta$ is indeed $\mu$-stationary. 

Now, for every non-trivial $g\in G$ we have
\begin{equation}\label{eq}
\int_X 1_{R_a(G_{x})}(g)d\nu(x)= \int_X \eta_x(\lambda_g)d\nu(x)= \eta(\lambda_g) = \eta\circ\bE(\lambda_g) = 0 ,    
\end{equation}
which implies that $g\notin R_a(G_{x})$ for $\nu$-a.e.\ $x\in X$. Since $G$ is countable, it follows $R_a(G_{x})$ is trivial for $\nu$-a.e. $x\in X$, in particular, there exists some $x_0$ with $R_a(G_{x_0})=\{e\}.$
\qed

\medskip

\noindent
{\it Proof of Corollary \ref{cor:lin}.}~
The fact that the existence of a $C^*$-simple stabilizer subgroup implies the simplicity of the crossed product is proven in~\cite{Oz}.

If $G\ltimes_rC(X)$ is simple, then by Theorem~\ref{main}, there exists some $x_0\in X$ with 
$G_{x_0}$ having a trivial amenable radical. By \cite{BKKO}, the classes of groups in question have the property that for all their subgroups, having a trivial amenable radical is equivalent to being $C^*$-simple, which proves the claim.
 \qed

\subsection*{$C^*$-simple groups}
Recall that a URS $Z$ of the group $G$ is a minimal component of the Chabauty space ${\rm Sub}(G)$ (URS stands for Uniformly Recurrent Subgroups, see~\cite{GW-URS}).
Kennedy proved~\cite{KenURS} that $G$ is $C^*$-simple if and only if it has no non-trivial amenable URS, namely the only minimal component of $Sub_{am}(G)$ is the trivial $\{\{e\}\}$. See also~\cite{le2018subgroup}.

Our results, when applied to the case of $C^*$-simple groups, yield a strengthening of Kennedy's result as follows.

\begin{theo}\label{thm:B}
Let $G$ be a countable $C^*$-simple group, and let $Z$ be an infinite URS of $G$. Then there exists an uncountable dense set of subgroups $Y \subseteq Z$ such that $R_a(H)=\{e\}$ for every $H\in Y$.
\begin{proof}
The idea of the proof is the same as the proof of Theorem~\ref{main}, but working with the reduced $C^*$-algebra $C_r^*(G)$ instead of the reduced crossed product, and the stationary theory of states in $S(C_r^*(G))$ we developed in~\cite{HK1} instead of its generalization of Amrutam-Ursu~\cite{AmrUrs}.

Since $G$ is $C^*$-simple, we can fix a generating $C^*$-measure $\mu\in Prob(G)$ \footnote{$\bigcup_n supp( \mu^n) = G$} such that  
the canonical trace $\tau_0$ is the unique $\mu$-stationary state in $S(C_r^*(G))$. 
With respect to that $\mu$, we can find a $\mu$-stationary measure on $Z$, which we denote by $\nu\in \pr(Z)$. 

For each $H\in Z$, define $\eta_H\in S(C_r^*(G))$ by $\eta_H(\lambda_g)=1_{R_a(H)}(g)$. Note that the map $H\mapsto\eta_H$ is Borel and equivariant (it is a special case of the setup of the proof of Theorem~\ref{main}).

Now let $\eta\in S(C_r^*(G))$ be defined by $\eta=\int_Z\eta_Hd\nu(H)$. It is straightforward to verify that $\eta$ is a $\mu$-stationary state, and hence, $\eta=\tau_0$. Again, by the countability of $G$ we conclude that for $\nu$-almost every $H$, $R_a(H)=\{e\}$ (see Eq.~\ref{eq}). Since $Z$ is minimal, and $\nu$ is a $\mu$-stationary in the classical sense, it is an atomless measure of full support (as $\mu$ is a generating measure). We conclude that $Y$, the set of subgroups for which $R_a(H)=\{e\}$ holds, is uncountable and dense. 
\end{proof}
\end{theo}

\begin{rem}
By \cite[Corollary 1.2]{bon2020realizing} every URS Z is \textit{a stabilizer URS}, that is, there exists a minimal compact $G$-space $X$, such that $Z={\{G_x : x\in X\}}\subset \Sub(G)$.

Also, when $G$ is $C^*$-simple, the reduced crossed product of any minimal action is simple \cite[Theorem 7.1]{BKKO}.
Thus, applying Theorem~\ref{main} for this $X$, we get a point $x_0\in X$ whose stabilizer $G_{x_0}$ has trivial amenable radical.

In fact, the proof of Theorem~\ref{main} showed that having a trivial amenable radial stabilizer holds for $\nu$-a.e.\ $x\in X$ with respect to a $\mu$-stationary measure in the sense of Amrutam-Ursu. However, such $\nu$,
unlike classical stationary measures, can have atomic support. In particular, one cannot guarantee uncountably many points with this property from this argument, hence our direct proof of Theorem~\ref{thm:B}. 
\end{rem}

We conclude this section by highlighting the recent work \cite{CGS-stationary26}, which establishes intriguing results on stationary random subgroups of countable groups $G$, naturally related to URSs of $G$.

\medskip

\noindent
{\bf Acknowledgment.} We thank Tattwamasi Amrutam and Dan Ursu for their helpful comments.

 \bibliographystyle{plain} \bibliography{bib} 
\end{document}